%% file: main.tex
\pdfoutput=1
\documentclass[letterpaper, 10 pt, conference]{ieeeconf}  % Comment this line out if you need a4paper

\IEEEoverridecommandlockouts                              % This command is only needed if 
\usepackage{graphicx,algorithm,algorithmic,amsthm,amsmath,amsfonts,verbatim,mathtools,enumerate,bm,hyperref,dsfont,pgfplots,xcolor,subfig}
\allowdisplaybreaks

\input defs.tex

\title{\LARGE \bf Alpha-Fair Routing in Urban Air Mobility with Risk-Aware Constraints
}

\author{Yue Yu, Zhenyu Gao, Sarah H.Q. Li, Qinshuang Wei, John-Paul Clarke, and Ufuk Topcu% <-this % stops a space
\thanks{Y. Yu, Q. Wei, and U. Topcu are with the Oden Institute for Computational Engineering and Sciences, The University of Texas at Austin, Austin, TX, 78712, USA (emails: yueyu@utexas.edu,\, qinshuang.wei@utexas.edu, utopcu@utexas.edu). Z. Gao and J-P. Clarke are with the Department of Aerospace Engineering and Engineering Mechanics, The University of Texas at Austin, Austin, TX, 78712, USA (email: zhenyu.gao@utexas.edu,\, johnpaul@utexas.edu). Sarah H.Q. Li is with the Autonomous Control Laboratory, ETH Z\"{u}rich, Z\"{u}rich, Switzerland (email: sarahli@control.ee.ethz.ch).}%
}

\begin{document}

\maketitle
\thispagestyle{empty}
\pagestyle{empty}

%%%%%%%%%%%%%%%%%%%%%%%%%%%%%%%%%%%%%%%%%%%%%%%%%%%%%%%%%%%%%%%%%%%%%%%%%%%%%%%%
\begin{abstract}
In the vision of urban air mobility, air transport systems serve the demands of urban communities by routing flight traffic in networks formed by vertiports and flight corridors. We develop a routing algorithm to ensure that the air traffic flow fairly serves the demand of multiple communities subject to stochastic network capacity constraints. This algorithm guarantees that the flight traffic volume allocated to different communities satisfies the \emph{alpha-fairness conditions}, a commonly used family of fairness conditions in resource allocation. It further ensures robust satisfaction of stochastic network capacity constraints by bounding the coherent risk measures of capacity violation. We prove that implementing the proposed algorithm is equivalent to solving a convex optimization problem. We demonstrate the proposed algorithm using a case study based on the city of Austin. Compared with one that maximizes the total served demands, the proposed algorithm promotes even distributions of served demands for different communities. 
\end{abstract}

%%%%%%%%%%%%%%%%%%%%%%%%%%%%%%%%%%%%%%%%%%%%%%%%%%%%%%%%%%%%%%%%%%%%%%%%%%%%%%%%
\input{introduction/introduction}

\input{network/network}

\input{routing/routing}
\input{numerical/numerical}

\input{conclusion/conclusion}
\input{appendix/appendix}


%%%%%%%%%%%%%%%%%%%%%%%%%%%%%%%%%%%%%%%%%%%%%%%%%%%%%%%%%%%%%%%%%%%%%%%%%%%%%%%%

%
%\section*{ACKNOWLEDGMENT} Y. Yu would like to thank Dr. Anushri Dixit for her helpful suggestion.

\bibliographystyle{IEEEtran}
\bibliography{IEEEabrv,reference}

\end{document}

%% file: defs.tex
\newcommand{\diag}{\mathop{\rm diag}}

\newcommand{\norm}[1]{\left\lVert#1\right\rVert}
\newcommand{\mnorm}[1]{{\left\vert\kern-0.25ex\left\vert\kern-0.25ex\left\vert #1 
    \right\vert\kern-0.25ex\right\vert\kern-0.25ex\right\vert}}

\newtheorem{proposition}{Proposition}
\newtheorem{assumption}{Assumption}

\newcommand{\ie}{{\it i.e.}}

%% file: introduction/introduction.tex
\section{Introduction}

Urban air mobility (UAM) aims to provide sustainable and accessible air transport services at low altitudes in metropolitan areas \cite{cohen2021urban2,biehle2022social}. Recent advancements in electrification, automation, and digitalization have enabled many technological innovations for aviation, such as those for electric vertical take-off and landing aircraft, unmanned aerial systems, and automated air traffic management. Together these innovations will facilitate a wide range of urban aviation tasks, such as passenger mobility, goods delivery, and emergency services.

One challenge in UAM is to allocate limited and uncertain amounts of network resources \emph{fairly} to different flights. On the one hand, the capacity of the transportation network in UAM, defined by its vertiports and flight corridors, is both limited---due to construction budget \cite{yu2023vertiport}, noise pollution regulations \cite{gao2023noise}, and safety concerns for navigation \cite{wei2023risk}---and uncertain---due to stochastic weather conditions, dynamic obstacles \cite{wu2022risk}, and emergencies. The air traffic routed through the network must robustly satisfy these uncertain capacity constraints. On the other hand, the network must support flights between multiple origins and destinations to serve the demands of different user groups.

The existing methods for air traffic management with fairness considerations mainly focus on allocating delays fairly among different flights. These methods allocate delays to different flights by minimizing certain deviations from a target allocation \cite{pelegrin2023urban,vossen2003general,ball2010ground}. This target allocation typically satisfies certain desired principles, such as the ration-by-distance principle, which prioritizes flights according to their flight distances \cite{vossen2003general,ball2010ground}, or the ration-by-schedule principle, which preserves the first-scheduled-first-served ordering of flights \cite{barnhart2012equitable}. Different planning methods minimize different deviation functions, such as the amount of overtaking and number of reversals in the ordering of flights \cite{bertsimas2016fairness,chin2021efficiency}, the amount by which the total delay exceeds the maximum expected delay along a route \cite{barnhart2012equitable,chin2021efficiency},  the proportion of the total delays allocated to an airline divided by its peak requests \cite{fairbrother2020slot,zografos2019bi}, and the lexicographical worst-case deviation among different flights \cite{jacquillat2018interairline}. Recent results in UAM also consider fairness under uncertainties \cite{sun2023fair}.   

However, these existing methods focus on fairness among the service providers, such as different airlines, rather than fairness among service receivers, such as users from different communities. In particular, they seek fair distribution of flight delays based on the flight distances \cite{vossen2003general,ball2010ground} or the schedules requested by different airlines \cite{barnhart2012equitable}, not the origin and destination of the flights. In the context of UAM, a paramount question regarding how to ensure fairness in serving the demands of multiple user communities---which is among the largest barriers to UAM's community acceptance~\cite{cohen2021urban2,biehle2022social}---is, to our best knowledge, still an open question. 

We develop a routing algorithm for UAM to fairly serve the demands of different communities subject to stochastic network capacity. This algorithm has two key features. First, it ensures that the air traffic volume serving the demands of different communities satisfies the \emph{alpha fairness conditions}, a general class of fairness conditions in resource allocation that includes many different fairness notions, such as proportional fairness, total delay fairness, and max-min fairness \cite{shakkottai2008network,xinying2023guide}. Second, it ensures robust satisfaction of stochastic capacities constraints of vertiports and flight corridors by bounding the coherent risk measures of capacity violation. We prove that implementing this algorithm is equivalent to solving a convex optimization problem with conic constraints. We demonstrate this algorithm using a case study based on the city of Austin. Compared with one that maximizes the total demands served, the proposed algorithm consistently promotes even distributions of demands served for different communities.

%% file: network/network.tex
\section{Alpha-fair routing in urban air mobility}
We introduce the alpha-fair routing problem in urban air mobility (UAM). The idea is to maximize the total utility function of the services in the form of payloads allocated to different communities subject to network flow constraints.

Throughout, we use the following notations. We let \(\mathbb{R}\), \(\mathbb{R}_{\geq 0}\), and \(\mathbb{N}\) denote the set of real numbers, nonnegative real numbers, and nonnegative integers, respectively. We let \(0_n\) and \(1_n\) denote the \(n\)-dimensional zero vector and all 1's vector, respectively. Given \(n_1, n_2\in\mathbb{N}\), we let \([n_1, n_2]\) denote the set of integers between \(n_1\) and \(n_2\). We let \(\Delta_n\subset\mathbb{R}^n\) denote the \(n\)-dimensional probability simplex.

\subsection{Nodes, links, and routes}
We model the transportation network using a directed graph \(\mathcal{G}\) composed of a set of nodes \(\mathcal{N}=\{1, 2, \ldots, n_n\}\), each of which corresponds to either a vertiport or a waypoint that marks the intersection of multiple flight corridors, and a set of links \(\mathcal{L}=\{1, 2, \ldots, n_l\}\), each of which corresponds to a flight corridor. Each link is an ordered pair of distinct nodes, where the first and second nodes are the ``tail" and ``head" of the link, respectively. A \emph{route} is a sequence of links connected in a head-to-tail fashion. We let \(\mathcal{R}=\{1, 2, \ldots, n_r\}\) denote a set of routes, which does not necessarily contain all possible routes in \(\mathcal{G}\). We will also use the notion of \emph{community}. Each route serves the demands of one or several of the total of \(n_c\in\mathbb{N}\) communities. 

\subsection{Incidence matrices}
We encode the topology of a graph \(\mathcal{G}\) using the notion of \emph{incidence matrices}, which, depending on its definition, describes which nodes each link connects, which links each route contains, and which communities' demands each route serves. In particular, we describe the incidence relationship between nodes and links using the node-link incidence matrix \(E\in\mathbb{R}^{n_n\times n_l}\). The entry \(E_{ij}\) in matrix \(E\) is associated with node \(i\) and link \(j\) as follows:
\begin{equation}\label{eqn: nl incidence}
     E_{ij}=\begin{cases}
    1, & \text{if node \(i\) is the head of link \(j\),}\\
    -1, & \text{if node \(i\) is the tail of link \(j\),}\\
    0, & \text{otherwise.}
    \end{cases}
\end{equation}
Similarly, we describe the incidence relationship between links and routes using the link-route incidence matrix \(F\in\mathbb{R}^{n_l\times n_r}\). The entry \(F_{ij}\) in matrix \(F\) is associated with link \(i\) and route \(j\) as follows:
\begin{equation}\label{eqn: lr incidence}
     F_{ij}=\begin{cases}
    1, & \text{if route \(j\) contains link \(i\),}\\
    0, & \text{otherwise.}
    \end{cases}
\end{equation}
Finally, we describe the incidence relationship between communities and routes using the community-route incidence matrix 
 \(H\in\mathbb{R}^{n_c\times n_r}\). The entry \(H_{ij}\) in matrix \(H\) is associated with community \(i\) and route \(j\) as follows:
\begin{equation}
     H_{ij}=\begin{cases}
    1, & \text{if route \(j\) serves the demand of community \(i\),}\\
    0, & \text{otherwise.}
    \end{cases}
\end{equation}

\subsection{Balanced flow constraints}
To ensure that the number of vehicles at each node is balanced, the total number of incoming vehicles must match the number of outgoing vehicles at each node. We denote the number of aircraft vehicles on different links per unit time using the \emph{vehicle flow vector} \(y\in\mathbb{R}^{n_l}\). In particular, the entry \([y]_{k}\) in vector \(y\) denotes the number of vehicles that use link \(k\). We formulate this constraint as follows: 
\begin{equation}\label{eqn: vehicle flow conserv}
    Ey=0_{n_n}, \, y\geq 0_{n_l}.
\end{equation}
If the constraint in \eqref{eqn: vehicle flow conserv} is violated, the vehicles will start to accumulate at some nodes, which cause unwanted crowding. %As a result, these node will have either too few or too many vehicles to sustain its daily operation.  

\subsection{Vehicle capacity constraints}
Whenever multiple routes share an overlapping link, the amount of payload transported on each route is jointly constrained by the number of vehicles scheduled on the overlapping link. To model this constraint, we first define the \emph{payload flow vector} \(z\in\mathbb{R}^{n_r}\) such that \(z_i\) denotes the number of payloads transported along route \(i\). Next, based on the link-route incidence matrix in \eqref{eqn: lr incidence}, the  vehicle flow vector and the payload flow vector jointly satisfy the following constraint:
\begin{equation}\label{eqn: vehicle capacity}
    Fz\leq y, \, z\geq 0_{n_r}.
\end{equation}

\subsection{Node and link capacity constraints}
To ensure collision-free operation, the number of vehicles on each link and node must be bounded. To model these constraints, we let \(
    c\in\mathbb{R}_{\geq 0}^{n_n}\) and \(d\in\mathbb{R}_{\geq 0}^{n_l}\)
denote the nominal node and link flow volume, respectively. In particular, the values of \(c_i\) and \(d_j\) denote the maximum number of vehicles that enter node \(i\) and link \(j\), respectively. Furthermore, we define a weighting matrix \(K\in\mathbb{R}^{n_n\times n_l}\) such that \(K=\max\{E, 0\}\)
\ie, each element in \(K\) is the elementawise maximum of the corresponding element in matrix \(E\) and zero.
Based on these matrices, we define the following \emph{node and link capacity constraints}:
\begin{equation}\label{eqn: nl cap constraint}
    Ky\leq (1+\epsilon) c,\, y\leq (1+\epsilon) d,  
\end{equation}
where \(\epsilon\in\mathbb{R}_{\geq 0}\) is a tolerance parameter. The constraints in \eqref{eqn: nl cap constraint} state that the number of vehicles on each node or link does not exceed the corresponding nominal value by more than \(100\epsilon\) percent. 

\subsection{Alpha-fair routing problem}
Given the network constraints, we aim to route the air traffic to serve the demands of multiple communities. Throughout we assume each community has sufficiently high demands and we aim to maximize the demands we serve for each community without violating network capacity constraints. To this end, we introduce the \emph{alpha-fair routing} approach, which determines the amount of payload allocated to each community via optimization. Let \(x\in\mathbb{R}^{n_c}_{\geq 0}\) denote the \emph{community demand vector} such that \(x_k\) denotes the payload demand we allocate to community \(k\). The idea of alpha-fair routing is to find an optimal allocation \(x^\star\in\mathbb{R}^{n_c}\) by solving the following optimization problem:
\begin{equation}\label{opt: alp routing}
    \begin{array}{ll}
    \underset{x, y, z}{\mbox{maximize}} & \sum_{k=1}^{n_c} \psi_k(x_k) \\
       \mbox{subject to}  & Ey=0_{n_n}, \, x = Hz,\, Fz\leq y,\\
       & Ky\leq (1+\epsilon) c,\, y\leq (1+\epsilon)d,\\
       &  y\geq 0_{n_l},\, z\geq 0_{n_r}\\
    \end{array}
\end{equation}
where \(\alpha\in [0, \infty)\) is a fairness parameter, and \(\psi_\alpha\) is the \emph{\(\alpha\)-utility function} defined as follows
\begin{equation}
    \psi_k(x_k)=\begin{cases}
        \log(x_k), & \text{if \(\alpha=1\)},\\
        \frac{x_k^{1-\alpha}}{1-\alpha}, & \text{if \(\alpha\geq 0, \alpha\neq 1\)}.
    \end{cases}
\end{equation}
% See Fig.~\ref{fig: alpha} for an illustration of \(\alpha\)-utility function. 

% \input{figs/alpha}

The goal of optimization~\eqref{opt: alp routing} is to find an optimal payload allocation that satisfies a set of \emph{fairness condition}. In particular, let \(x^\star\in\mathbb{R}^{n_c}\) be  optimal for optimization~\eqref{opt: alp routing}. Then one can prove that, for any \(x\in\mathbb{R}^{n_c}\) such that \((x, y, z)\) satisfies the constraints in \eqref{opt: alp routing} for some \(y\in\mathbb{R}^{n_l}\) and \(z\in\mathbb{R}^{n_r}\), we must have  
\begin{equation}\label{eqn: alp fair cond}
    \sum_{k=1}^{n_c}\frac{x_k-x_k^\star}{(x^\star_k)^\alpha}\leq 0.
\end{equation}
The conditions in \eqref{eqn: alp fair cond} are known as the \emph{\(\alpha\)-fairness conditions} \cite{shakkottai2008network}. They include many popular notions of fairness in the network resource allocation literature, including proportional fairness, total delay fairness, and max-min fairness. We refer the interested readers to \cite[Sec. 2.4]{shakkottai2008network} and the references therein for further details on alpha-fairness.

%% file: routing/routing.tex
\section{Risk-aware alpha-fair routing}
One drawback of the routing problem in \eqref{opt: alp routing} is that it ignores the effects of uncertainty in air mobility. In the following, we introduce a variation of the routing problem in \eqref{opt: alp routing} that accounts for uncertain network capacity, which occurs due to, for instance, weather conditions and emergency shutdowns. To this end, we denote the capacity parameter as \(\omega\coloneqq \begin{bmatrix}
c^\top & d^\top 
\end{bmatrix}^\top\).
Furthermore, we make the following assumption on \(\omega\).   

\begin{assumption}\label{asp: rand cap}
The parameter vector \(\omega\coloneqq \begin{bmatrix}
c^\top & d^\top 
\end{bmatrix}^\top\) is a discrete random variable with \(N\in\mathbb{N}\) possible outcomes. Its sample space is given by
\(\{\omega^i\coloneqq \begin{bmatrix}
(c^i)^\top & (d^i)^\top 
\end{bmatrix}^\top\}_{i=1}^N\), where \(c^i\in\mathbb{R}_{\geq 0}^{n_n}\) and \(d^i\in\mathbb{R}_{\geq 0}^{n_l}\) for all \(i\in[1, N]\).
\end{assumption}
In the following, we will discuss how to modify the capacity constraints in \eqref{eqn: nl cap constraint} and the corresponding routing problem in \eqref{opt: alp routing} under Assumption~\eqref{asp: rand cap}

\subsection{Alpha-fair routing with risk-aware constraints}

If \((c, d)=(c^i, d^i)\) for some \(i\in[1, d]\), then the capacity constraints in \eqref{eqn: nl cap constraint} is equivalent to a constraint on the optimal value of a linear program. In this case, one can show that \eqref{eqn: nl cap constraint} is equivalent to 
\begin{equation}\label{eqn: h star i constraint}
    h_i^\star\leq \epsilon,
\end{equation}
where \(h_i^\star\in\mathbb{R}\) is the optimal value of the following linear program
\begin{equation}\label{opt: h star i}
    \begin{array}{ll} 
    \underset{h_i}{\mbox{minimize}} & h_i \\
    \mbox{subject to}     & Ky\leq (1+h_i)c^i,\, y\leq (1+ h_i) d^i, \, h_i\geq 0.
    \end{array}
\end{equation}
To see the equivalence between \eqref{eqn: nl cap constraint} and \eqref{eqn: h star i constraint}, notice that \eqref{eqn: nl cap constraint} implies that \(h_i=\epsilon\) satisfies the constraints in \eqref{opt: h star i}, hence \eqref{eqn: h star i constraint} must hold. On the other hand, if \eqref{eqn: h star i constraint} holds, then letting \(h_i=h_i^\star\) in the constraints in \eqref{opt: h star i} implies that \eqref{eqn: nl cap constraint}.

To constrain the capacity violation under Assumption~\ref{asp: rand cap}, we define a scalar-valued random variable \(\eta\) with a finite sample space \(\{h^\star_1, h^\star_2, \ldots, h_N^\star\}\), where \(h^\star_i\) is the optimal value of the deterministic linear program in \eqref{opt: h star i}. In addition, we introduce a family of functions, denoted by \(\rho_\delta:\mathbb{R}^{n_l}\to\mathbb{R}\) for some parameter \(\delta\in(0, 1)\), such that   
 \begin{equation}\label{eqn: crm}
\rho_\delta(y) =\max_{q\in\mathbb{Q(\delta)}}\, q^\top h^\star,
\end{equation} 
where set \(\mathbb{Q}(\delta)\subset\Delta_N\) is a closed and convex set that satisfies the following assumption.
\begin{assumption}\label{asp: risk envelope}
    There exists \(p\in\Delta_N\) and a closed, convex, and proper function \(g_{\delta}:\mathbb{R}^N\to\mathbb{R}\cup \{\infty\}\) such that \(g_\delta(p)<0\) and \(\mathbb{Q}(\delta)\coloneqq \left\{q\in\mathbb{R}^N| g_{\delta}(q)\leq 0\right\} \cap \Delta_N\). Furthermore, \(g_{\delta_1}(q)>g_{\delta_2}(q)\) for all \(q\in\Delta_N\) and \(0<\delta_1<\delta_2<1\).
\end{assumption}

Function \(g_\delta\) provides a \emph{coherent risk measure} of random variable \(\eta\). In particular, the representation theorem states that a coherent risk measure of a random variable \(\eta\)---which is characterized by the monotonicity, translational invariance, positive homogeneity, subadditivity properties \cite{artzner1999coherent}---is equivalent to the maximum expectation of this variable over a convex and closed set of probability distributions \cite[Prop. 4.1]{artzner1999coherent}. Hence, under Assumption~\ref{asp: risk envelope}, \eqref{eqn: crm} implies that \(g_\delta(y)\) is is a coherent risk measure of \(\eta\), where \(\delta\in[0, 1]\) is an uncertainty parameter: the volume of set \(\mathbb{Q}(\delta)\) increases as \(\delta\) increases. Many common risk measures satisfy Assumption~\ref{asp: risk envelope}. See Appendix for some examples.

Based on function \eqref{eqn: crm}, we propose to replace the stochastic constraint in \eqref{eqn: nl cap constraint} with the following:
\begin{equation}\label{eqn: risk constraint}
    \rho_\delta(y)\leq \epsilon.
\end{equation}
Under Assumption~\ref{asp: rand cap}, if \(\mathbb{Q}(\delta)=\{p\}\), then \eqref{eqn: risk constraint} implies that the constraints in \eqref{eqn: nl cap constraint} only hold in expectation. On the other hand, if \(\mathbb{Q}(\delta)=\Delta_N\), then \eqref{eqn: h star i constraint} holds for all \(i\in[1, d]\). By ranging \(\delta\) within the interval \([0, 1]\), the constraint in \eqref{eqn: risk constraint} ensures different tradeoffs between a risk-neutral and a risk-averse variation of the stochastic constraints in \eqref{eqn: nl cap constraint}.  

By replacing the constraints in \eqref{eqn: nl cap constraint} with its risk-aware variation in \eqref{eqn: crm}, we  obtain a risk-aware variation of the routing problem in \eqref{opt: alp routing}, given as follows:
\begin{equation}\label{opt: crm routing}
    \begin{array}{ll}
    \underset{x, y, z}{\mbox{maximize}} & \sum_{k=1}^{n_c} \psi_k(x_k) \\
       \mbox{subject to}  & Ey=0_{n_n}, \, x = Hz,\, Fz\leq y,\\
       & \rho_\delta(y)\leq \epsilon,\, y\geq 0_{n_l},\, z\geq 0_{n_r}.    \end{array}
\end{equation}
 
\subsection{Simplification of risk-aware constraints}

Optimization~\eqref{opt: crm routing} seems challenging to solve: it contains an outer layer optimization over the network flow variables and an inner layer optimization that comes from the definition of the coherent risk measure in \eqref{eqn: crm}. In the following, we show that this nested two-layer optimization is equivalent to a single-level convex optimization using duality theory. To this end, we make the following assumption on set \(\mathbb{Q}\) in \eqref{eqn: crm}.

Assumption~\ref{asp: risk envelope} states that set \(\mathbb{Q}\) is the intersection of the sublevel set of a convex function and the probability simplex. Many popular risk measures---such as the conditional value-at-risk, entropic value-at-risk, and total variation distance---satisfy this assumption. See the Appendix for the detailed formula of some popular risk measures. 
 
Based on Assumption~\ref{asp: risk envelope}, the value of the coherent risk measure in \eqref{eqn: crm} is the optimal value of not only the maximization problem in \eqref{eqn: crm}, but also its dual problem, as shown by the following proposition.

\begin{proposition}\label{prop: crm conic}
Suppose that Assumption~\ref{asp: rand cap} and Assumption~\ref{asp: risk envelope} hold. Given \(\delta\in(0, 1)\) and \(y\in\mathbb{R}^{n_l}_{\geq 0}\), the value of \(\rho_{\delta}(y)\) equals the optimal value of the following optimization problem
\begin{equation}\label{opt: crm}
\begin{array}{ll}
    \underset{w, \lambda, \nu, h}{\mbox{minimize}}  & \lambda g_\delta^*(\frac{1}{\lambda}w)-\nu \\
    \mbox{subject to} & Ky\leq (1+h_i)c^i,\, y\leq (1+h_i) d^i,\, h\geq 0_N,\\
    & w-\nu1_N\geq h,\, \lambda\geq 0,\, i\in[1, N],
\end{array}
\end{equation}
where \(g_\delta^*:\mathbb{R}^N\to\mathbb{R}\cup\{\infty\}\) is the convex conjugate function of \(g_\delta\), \ie, \(g_\delta^*(w)\coloneqq \sup_q \, q^\top w-g_\delta(w)\)  for all \(w\in\mathbb{R}^N\). Furthermore, we define \(0\cdot g_\delta^*(\frac{1}{0}w)\coloneqq \infty\) for all \(w\in\mathbb{R}^N\).
\end{proposition}

\begin{proof}
We first show that the dual problem of the optimization in \eqref{eqn: crm} is as follows:
\begin{equation}\label{opt: crm min}
\begin{array}{ll}
    \underset{\lambda, \nu, w}{\mbox{minmize}}  & \lambda g_\delta^*(\frac{1}{\lambda}w)-\nu\\
    \mbox{subject to} & w-\nu1_N\geq h^\star, \, \lambda\geq 0.
\end{array}
\end{equation}
Notice that we implicitly assume that \(\lambda>0\) in \eqref{opt: crm min}. The Lagrangian of optimiztaion~\eqref{opt: crm min} is given by \(L(w, \lambda, \nu, q, \sigma)= \textstyle \lambda g_\delta^*(\frac{1}{\lambda}w)+ q^\top (h^\star+\nu1_N-w)-\sigma\lambda-\nu.\).
The dual problem of optimization~\eqref{opt: crm min} is as follows:
\begin{equation}\label{opt: crm dual 1}
\begin{array}{ll}
    \underset{ q, \sigma}{\mbox{maxmize}}  & \underset{ w, \lambda, \nu}{\mbox{min}}\,L(w, \lambda, \nu, q, \sigma)\\
    \mbox{subject to} &  q\geq 0_N, \, \sigma\geq 0.
\end{array}
\end{equation}
For any \(\lambda>0\), we can show that 
\begin{equation}\label{eqn: w min}
\textstyle  \underset{ w }{\mbox{min}} \, \lambda g_\delta^*(\frac{1}{\lambda} w)-q^\top w=\textstyle-\underset{ u }{\mbox{max}} \, \lambda q^\top u-\lambda g_\delta^*(u)=-\lambda g_\delta(q),
\end{equation}
where the last step is due to the fact that \((g_\delta^*)^*=g_\delta\) when \(g_\delta\) under Assumption~\ref{asp: risk envelope} \cite[Thm.12.2]{rockafellar2015convex}.
By using \eqref{eqn: w min} we can show the following:
\begin{equation}\label{eqn: dual func}
    \begin{aligned}
        &\underset{ \lambda, \nu, w}{\mbox{min}}\,L(w, \lambda, \nu, q, \sigma)\\
        %&=\textstyle \underset{\lambda, \nu}{\mbox{min}}\, \underset{w}{\mbox{min}}\,\textstyle \lambda g_\delta^*(\frac{1}{\lambda}w)+ q^\top (h^\star+\nu1_N-w)-\sigma\lambda-\nu\\
        &= \underset{\lambda, \nu}{\mbox{min}}\, -\lambda g_\delta(q)+ q^\top (h^\star+\nu1_N)-\sigma\lambda-\nu\\
        &=\begin{cases}
            q^\top h^\star, & \text{if } g_\delta(q)=-\sigma, \, q^\top 1_N=1,\\
            -\infty, & \text{otherwise.}
        \end{cases}
    \end{aligned}
\end{equation}

Substituting \eqref{eqn: dual func} into \eqref{opt: crm dual 1} gives
\begin{equation}\label{opt: crm dual 2}
\begin{array}{ll}
    \underset{q}{\mbox{maxmize}}  & \textstyle q^\top h^\star\\
    \mbox{subject to} &  q^\top 1_N=1, \, q\geq 0,\, g(q)\leq 0.
\end{array}
\end{equation}
Notice that, under  Assumption~\ref{asp: risk envelope}, optimization~\eqref{opt: crm dual 2} is exactly the optimization in \eqref{eqn: crm}. Since Assumption~\ref{asp: rand cap} and Assumption~\ref{asp: risk envelope} together imply that optimization~\eqref{opt: crm dual 2} is strictly feasible if \(q=p\), we conclude that optimization~\eqref{opt: crm min} and optimization~\eqref{opt: crm dual 2} have the same optimal value.

Finally, suppose that \((w, \lambda, \nu)\) satisfies the constraints in optimization~\eqref{opt: crm min}. Since \(h^\star\) satisfies the constraints in \eqref{opt: h star i} for all \(i\in[1, N]\), \((w, \lambda, \nu, h^\star)\) satisfies the constraints in optimization~\eqref{opt: crm}. Next, suppose that \((w, \lambda, \nu, h)\) is optimal for optimization~\eqref{opt: crm}. Since \(h_i^\star\) is the optimal value of optimization~\eqref{opt: h star i} for all \(i\in[1, N]\), we necessarily have 
\begin{equation*}
    w-\nu1_N\geq h\geq h^\star.
\end{equation*}
Hence \((w, \lambda, \nu)\) satisfies the constraints of optimization~\eqref{opt: crm min}. Therefore there exists \(h\in\mathbb{R}^n\) such that \((w, \lambda, \nu, h)\) satisfies the constraints in optimization~\eqref{opt: crm min} if and only if \((w, \lambda, \nu)\) satisfies the constraints in optimization~\eqref{opt: crm}. Since optimization~\eqref{opt: crm} and optimization~\eqref{opt: crm min} have the same objective function, we conclude that the optimal values of the two optimization problems are the same. 
\end{proof}

Proposition~\ref{prop: crm conic} shows that any upper bound on the coherent risk measure in \eqref{eqn: crm} is an upper bound on the optimal value of a convex minimization problem. Based on this observation, the following proposition shows that one can merge said minimization together with the outer-layer maximization in \eqref{opt: crm routing} by introducing additional variables and constraints.

\begin{proposition}\label{prop: crm routing conic}
Suppose that Assumption~\ref{asp: rand cap} holds. 
Optimization~\eqref{opt: crm routing} is equivalent to the following optimization problem
\begin{equation}\label{opt: crm routing conic}
    \begin{array}{ll}
    \underset{x, y, z, w, \lambda, \nu, h}{\mbox{maximize}} & \sum_{k=1}^{n_c} \psi_k(x_k) \\
       \mbox{subject to}  & Ey=0_{n_n}, \, x = Hz,\, Fz\leq y,\\
       &Ky\leq (1+h_i)c^i,\, y\leq (1+h_i)d^i,\\
    &w-\nu1_N\geq h,\, y\geq 0_{n_l}, \, z\geq 0_{n_r},\, \,  h\geq 0_N,\\
    &\lambda g_\delta^*(\frac{1}{\lambda}w)-\nu\leq \epsilon,\,\lambda\geq 0,\, i\in[1, N].
    \end{array}
\end{equation}
\end{proposition}
\begin{proof}
    If \((x, y, z, w, \nu, \lambda, h)\) satisfies the constraints in \eqref{opt: crm routing conic}, then \((w, \lambda, \nu, h)\) satisfies the constraints in \eqref{opt: crm}. Furthermore, Proposition~\eqref{prop: crm conic} implies that
    \begin{equation}
        \textstyle \epsilon \geq \lambda g_\delta^*(\frac{1}{\lambda}w)-\nu\geq \rho_{\delta}[\zeta_{\mathbb{D}}(y)]. 
    \end{equation}
    Hence \((x, y, z)\) satisfies the constraints in \eqref{opt: crm routing}. On the other hand, if  \((x, y, z)\) satisfies the constraints in \eqref{opt: crm routing}, then Proposition~\ref{prop: crm conic} implies that there exists \(( w, \lambda, \nu, h)\) such that 
    \begin{equation}
         \textstyle \lambda g_\delta^*(\frac{1}{\lambda}w)-\nu= \rho_{\delta}[\zeta_{\mathbb{D}}(y)]\leq \epsilon. 
    \end{equation}
    In other words, \((x, y, z, w, \lambda, \nu, h)\) satisfies the constraints in \eqref{opt: crm routing conic}. Therefore, \((x, y, z)\) satisfies the constraints in \eqref{opt: crm routing} if and only if there exists \(w, \lambda, \nu, h\) such that \((x, y, z, w, \lambda, \nu, h)\) satisfies the constraints in \eqref{opt: crm routing conic}. Since optimization~\eqref{opt: crm routing} and \eqref{opt: crm routing conic} have the same objective function, we complete the proof.
\end{proof}

Proposition~\ref{prop: crm routing conic} shows that the risk-aware routing problem in \eqref{opt: crm routing} is equivalent to the constrained convex optimization problem in \eqref{opt: crm routing conic}. For come common risk measures, the constraints in \eqref{opt: crm routing conic} are either linear or exponential cone constraints. See Appendix for details.

%% file: numerical/numerical.tex
\section{Numerical Experiments}

We demonstrate the proposed routing algorithm using a case study based on a UAM network in the city of Austin, illuatrated by Fig.~\ref{subfig: Austin network}. This network serves the demands of \(n_c=46\) communities, each corresponding to a zip code area in the city of Austin. It contains \(n_n=17\) nodes, each corresponding to a vertiport; \(n_l=72\) links, each corresponding to a flight corridor; and \(n_r=200\) candidate routes, each containing at most five connected flight corridors.  We choose the locations of the nodes to be city centers or subcenters that are geographically and societally reconfigurable into vertiports. To simplify air traffic management in UAM, we deploy the links such that there are no intersections between different links. We select the 200 routes based on the public transportation plans and the function of different communities. We choose the capacities of the vertiports and flight corridors based on utilization forecast and safety and noise considerations. We also consider a uniform capacity reduction of \(20\%\) and \(40\%\) for all vertiports and flight corridors, which occurs with probability \(0.3\) and \(0.2\), respectively. 

We aim to route the network flow of vehicles in the Austin network to fairly serve the demands of all 46 communities subject to constraints on the coherent risk measures of the stochastic link and node capacity violation. We illustrate the routing results obtained by solving optimization~\eqref{opt: crm routing conic} using the Frank-Wolfe method \cite{jaggi2013revisiting} (terminated after 100 iterations) in Fig.~\ref{fig: Austin maps} and Fig.~\ref{fig: different delta}. We compare them with the results of maximizing the total community demand served, which is obtained by solving optimization~\eqref{opt: crm routing conic} with \(\psi_k(x_k)=x_k\) for all \(k\in[1, n_c]\). These results show that optimizing the alpha-utility function promotes even distribution of the demands served for different communities.

\input{numerical/Austin}
\input{numerical/scatter}

%% file: numerical/Austin.tex
\begin{figure*}[t]
	\centering
 \subfloat[Austin network\label{subfig: Austin network}]{
        \includegraphics[width=0.32\textwidth]{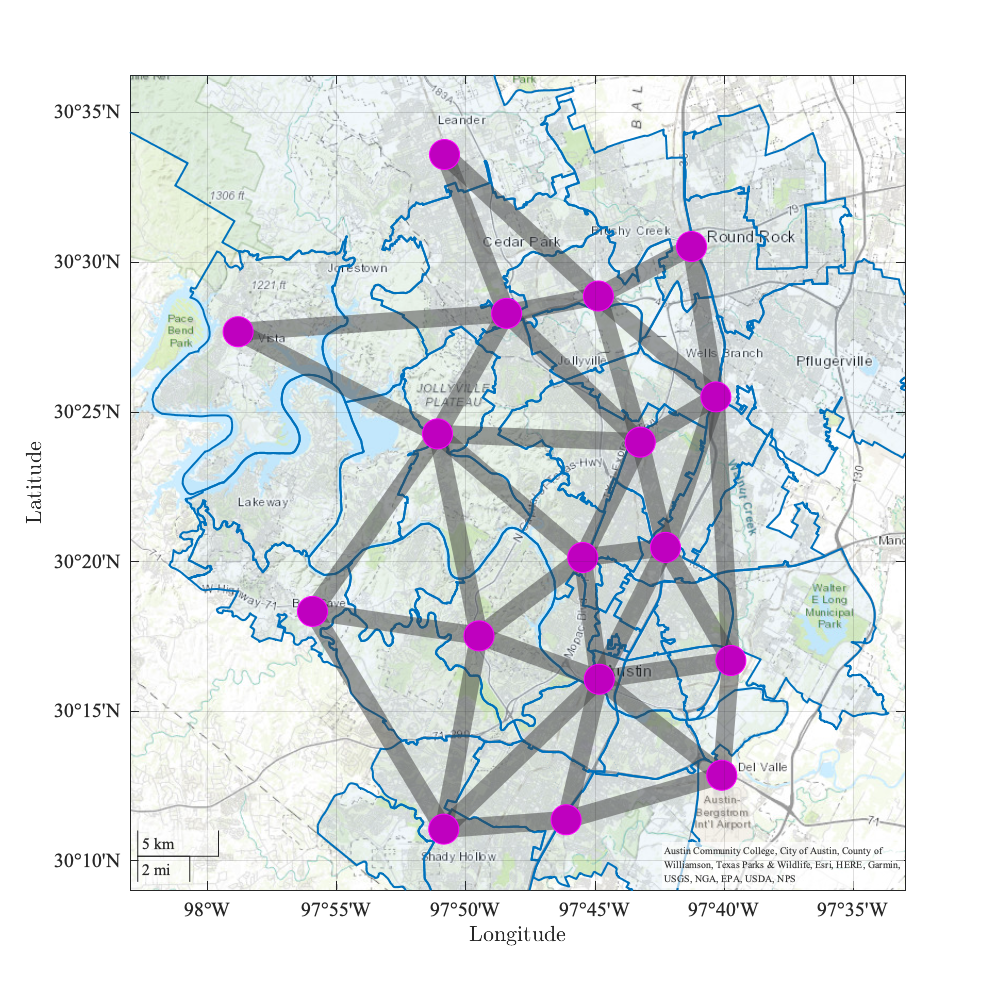}
        }
        \hfill
        \subfloat[Maximizing the total demand served \label{subfig: Austin maxsum}]{\includegraphics[width=0.32\textwidth]{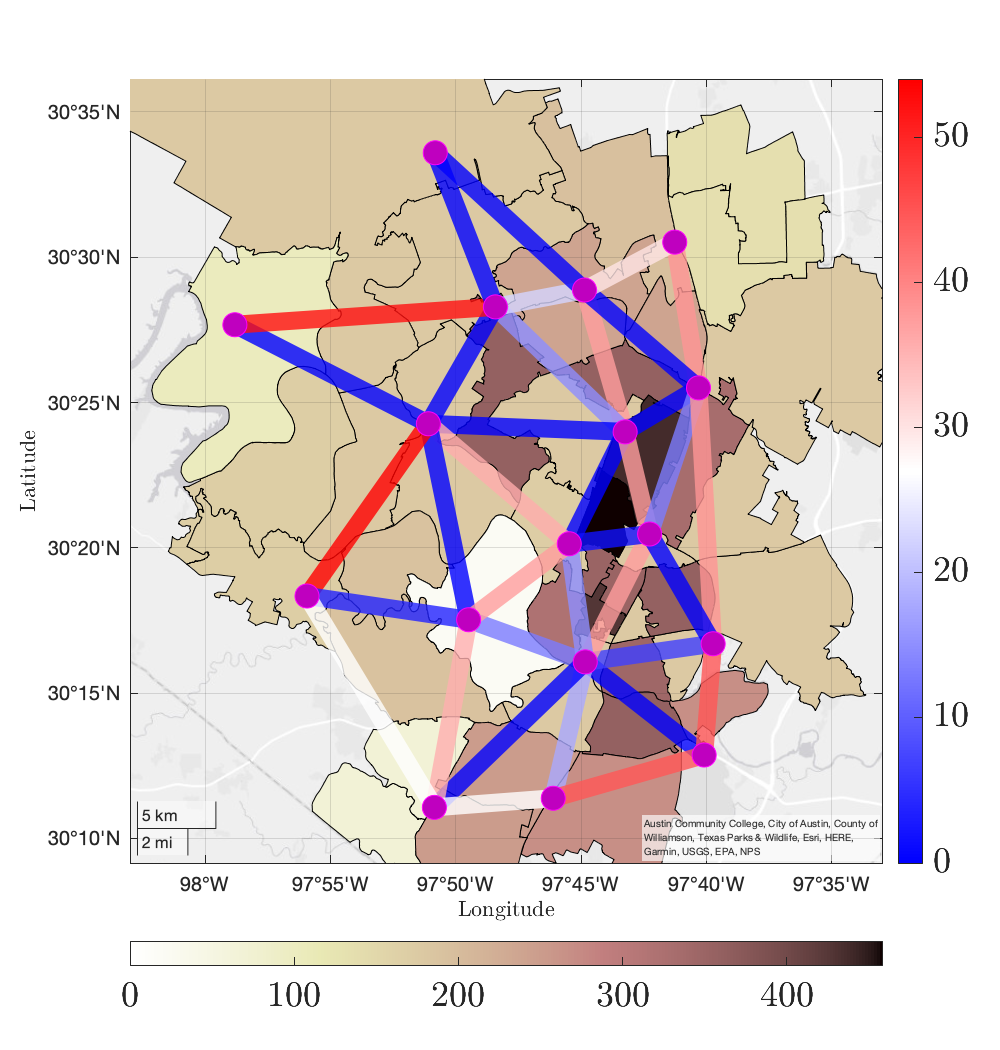}}
        \hfill
        \subfloat[Ensuring alpha-fairness\label{subfig: Austin alpha fair}]{\includegraphics[width=0.32\textwidth]{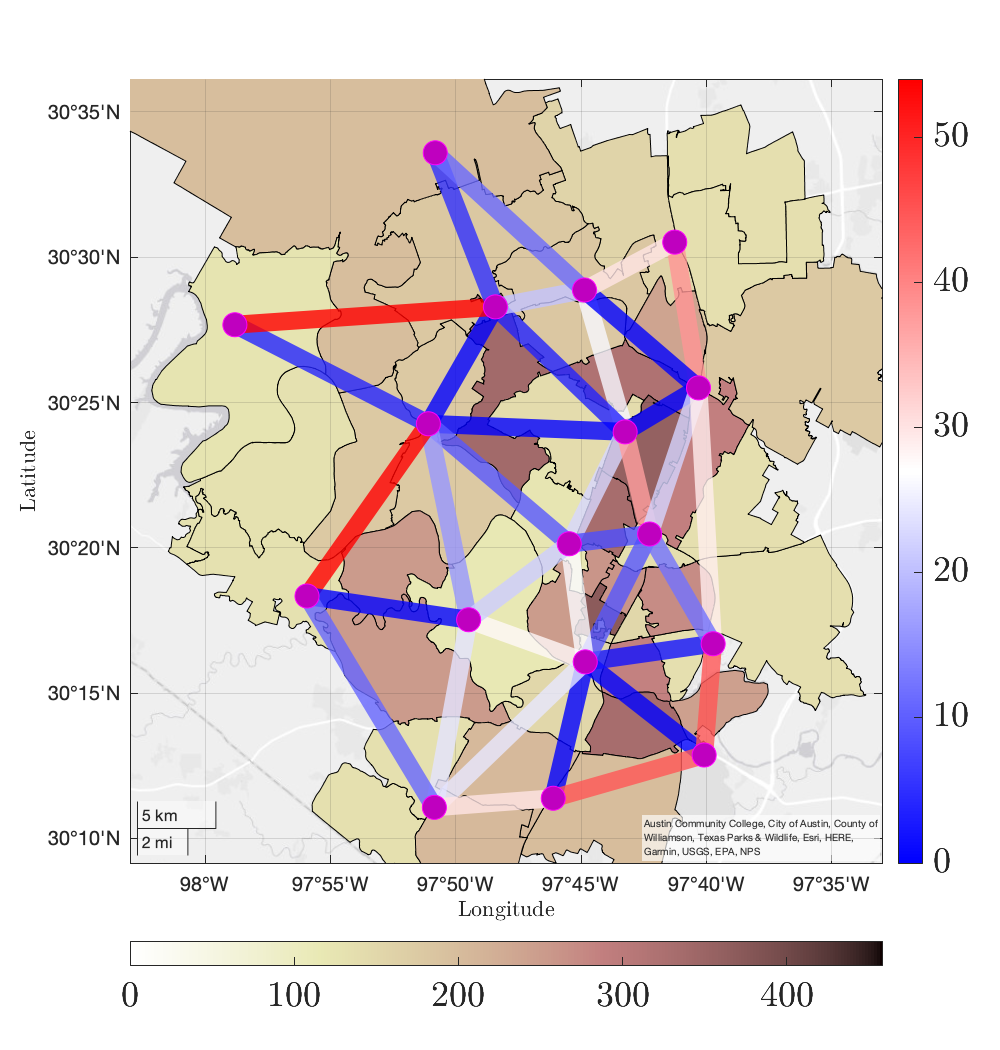}}
	\caption{The Austin network (left), and the average number of flights per hour for different links and communities when maximizing the total community demand served (middle) and ensuring alpha-fairness (right) subject to capacity constraints based on the conditional value-at-risk.}
	\label{fig: Austin maps}
\end{figure*}

%% file: numerical/scatter.tex
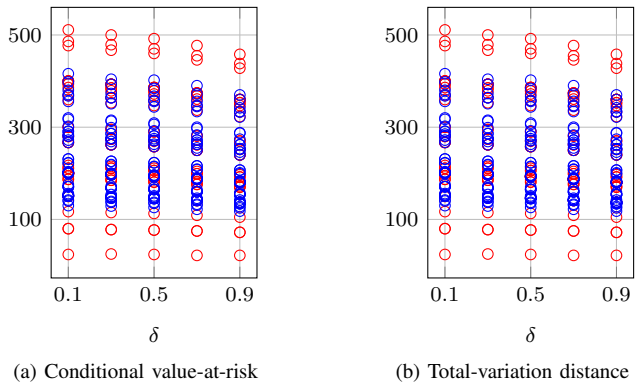
\begin{figure}[!ht]
\centering
     \subfloat[Conditional value-at-risk\label{subfig: cvar different del}]{%
       \begin{tikzpicture}
    \begin{axis}[xlabel={{\footnotesize \(\delta\)}},width=0.5\columnwidth,
            height =0.6\columnwidth,
            xtick={0.1, 0.5, 0.9},
            ytick={100, 300, 500},
            ytick style={draw=none},
             ymajorgrids = true,
            xmajorgrids = true,
               yticklabel style = {font=\footnotesize},
            xticklabel style = {font=\footnotesize}
            ]
       \addplot[
            scatter/classes={
                1={mark=o,blue},
                2={mark=o,red}
                },
                scatter, only marks,
                scatter src=explicit symbolic,
                nodes near coords*={},]
         table[x=x,y=y,meta=label]
            {numerical/data/cvar_maxsum.dat};
    
        \addplot[
            scatter/classes={
                1={mark=o,blue},
                2={mark=o,red}
                },
                scatter, only marks,
                scatter src=explicit symbolic,
                nodes near coords*={},]
         table[x=x,y=y,meta=label]
            {numerical/data/cvar_alp10.dat};

    \end{axis}
\end{tikzpicture}
     }
     \hfill
     \subfloat[Total-variation distance\label{subfig: tv different del}]{%
       \begin{tikzpicture}
    \begin{axis}[xlabel={{\footnotesize \(\delta\)}},width=0.5\columnwidth,
            height =0.6\columnwidth,
            xtick={0.1, 0.5, 0.9},
            ytick={100, 300, 500},
            ytick style={draw=none},
             ymajorgrids = true,
            xmajorgrids = true,
               yticklabel style = {font=\footnotesize},
            xticklabel style = {font=\footnotesize}
            ]
       \addplot[
            scatter/classes={
                1={mark=o,blue},
                2={mark=o,red}
                },
                scatter, only marks,
                scatter src=explicit symbolic,
                nodes near coords*={},]
         table[x=x,y=y,meta=label]
            {numerical/data/cvar_maxsum.dat};
    
        \addplot[
            scatter/classes={
                1={mark=o,blue},
                2={mark=o,red}
                },
                scatter, only marks,
                scatter src=explicit symbolic,
                nodes near coords*={},]
         table[x=x,y=y,meta=label]
            {numerical/data/cvar_alp10.dat};

    \end{axis}
\end{tikzpicture}
     }
     \caption{The average number of flights per hour assigned to different communities subject to capacity constraints based on different risk measures and different values of parameter \(\delta\). The red and blue markers show the results when maximizing the total demand served and ensuring alpha-fairness, respectively. }
     \label{fig: different delta}
\end{figure}

%% file: conclusion/conclusion.tex
\section{Conclusion}
We develop a routing algorithm to fairly serve the demands of multiple communities in UAM subject to risk-aware constraints. We ensure the fairness of the demands served for different communities by maximizing the sum of alpha-utility functions, and the robust satisfaction of stochastic link and node capacity constraints via bounding the coherent risk measures of capacity violation. We demonstrate our results using a UAM network based on the city of Austin.

However, our current results still have several limitations. For example, the proposed algorithm does not consider dynamic changes in link and node capacity to time-dependent noise regulations. For future work, we plan to extend our results to dynamic routing as well as simultaneous network design and traffic routing in UAM.

%% file: appendix/appendix.tex
\section*{Appendix: Function Common risk measures}
We provide some common examples of formulas of the function \(g_\delta\) in Assumption~\ref{asp: risk envelope} for some common coherent risk measures, along with its convex conjugate function, which appears in optimization~\eqref{opt: crm routing conic}.
\subsection{Conditional value-at-risk}
For conditional value-at-risk, the function \(g\) in Assumption~\ref{asp: risk envelope} takes the following form: 
\begin{equation}
    g_\delta(q)\coloneqq \norm{\diag(p)^{-1}q}_\infty-1/(1-\delta)
\end{equation}
We can show, for any \(r\in\mathbb{R}^N\), that
\begin{equation}\label{eqn: cvar g*}
\begin{aligned}
    g_\delta^*(r) &= 1/(1-\delta)+\sup_q \, r^\top q -\norm{\diag(p)^{-1}q}_\infty\\
    &= 1/(1-\delta) +\sup_u \, r^\top \diag(p) u-\norm{u}_\infty\\
    &=\begin{cases}
    1/(1-\delta), & \text{if } \norm{\diag(p)r}_1\leq 1,\\
    \infty, & \text{otherwise.}
    \end{cases}
\end{aligned}   
\end{equation}

\begin{comment}

 to the following one:
\begin{equation*}
    \begin{array}{ll}
    \underset{x, y, z, w, \lambda, \nu, h}{\mbox{maximize}} & \sum_{k=1}^{n_c} \psi_k(x_k) \\
       \, \mbox{subject to}  & Ey=0_{n_n}, \, x = Hz,\, Fz\leq y,\\
       &Ky\leq (1+h_i)c^i,\, y\leq (1+h_i)d^i,\\
    &y\geq 0_{n_l}, \, z\geq 0_{n_r},\, \,  h\geq 0_N,\,\lambda\geq 0,\\
    & w-\nu1_N\geq h,\, -\nu+\frac{1}{1-\delta}\lambda\leq \epsilon,\\
    &  \norm{\diag(p) w }_1\leq \lambda\, i\in[1, N].
    \end{array}
\end{equation*}
    
\end{comment}

\subsection{Entropic value-at-risk}
For entropic value-at-risk, the function \(g\) in Assumption~\ref{asp: risk envelope} takes the following form: 
\begin{equation}
    g_\delta (q) \coloneqq \begin{cases} 
    \sum_{i=1}^N q_i\ln(q_i/p_i)+\ln(1-\delta), & \text{if } q\in\Delta_N,\\
    \infty, & \text{otherwise.}
    \end{cases}
\end{equation}
Given \(r\in\mathbb{R}^N\), we have
\begin{equation}\label{eqn: KL conjugate}
  \textstyle   g_\delta^*(r) = -\ln(1-\delta)+\sup_{q\in\Delta_N} \, r^\top q- \sum_{i=1}^N q_i\ln(q_i/p_i)
\end{equation}
Using the KKT conditions, we can show that the supreme in \eqref{eqn: KL conjugate} is attained if there exists \(\beta\in\mathbb{R}\) such that \(0=r_i-\ln(q_i/p_i)-\beta-1\) for all \(i\in[N]\) and \(1=q^\top 1_N,\). Substituting these conditions into \eqref{eqn: KL conjugate} shows that
\begin{equation}\label{eqn: evar g*}
    g_\delta^*(r) = -\ln(1-\delta)+\ln(p^\top \exp(r))
\end{equation}
where \(\exp(r)\) is the elementwise exponential of \(r\). In this case, one can show that \(g_\delta^*(r)\) is the optimal value of an exponential cone program. See \cite[Sec. 4.1.2]{dixit2022risk} for a related discussion. 

\begin{comment}

Notice that, if we  substitute \eqref{eqn: cvar g*} into \eqref{opt: crm routing conic}, optimization~\eqref{opt: crm routing conic} reduces to the following \emph{exponential cone program}:
\begin{equation*}
    \begin{array}{ll}
    \underset{x, y, z, w, u, \lambda, \nu,\tau,  h}{\mbox{maximize}} & \sum_{k=1}^{n_c} \psi_k(x_k) \\
       \enskip \,\mbox{subject to}  & Ey=0_{n_n}, \, x = Hz,\, Fz\leq y,\\
       &Ky\leq (1+h_i)c^i,\, y\leq (1+h_i)d^i,\\
    &y\geq 0_{n_l}, \, z\geq 0_{n_r},\, \,  h\geq 0_N,\,\lambda\geq 0,\\
    & w-\nu 1_N\geq h,\, -\nu+\tau-\ln(1-\delta)\lambda\leq \epsilon,\\
    & \lambda \exp(\frac{1}{\lambda} (w_i-\tau))\leq u_i,\, \\
    &p^\top u\leq \lambda,\, i\in[1, N].
    \end{array}
\end{equation*}
    
\end{comment}

\subsection{Total variation distance}
For the total variance distance \cite{dixit2022distributionally}, the function \(g\) in Assumption~\ref{asp: risk envelope} takes the following form: 
\begin{equation}
    g_\delta(q)\coloneqq \norm{q-p}_1-2\delta.
\end{equation}
Given \(r\in\mathbb{R}^N\), we can show that
\begin{equation}\label{eqn: tvd g*}
\begin{aligned}
    g_\delta^*(r) &= 2\delta+r^\top p+\sup_q \, r^\top (q-p) -\norm{q-p}_1\\
    &=\begin{cases}
    2\delta +r^\top p, & \text{if } \norm{r}_\infty\leq 1,\\
    \infty, & \text{otherwise.}
    \end{cases}
\end{aligned}   
\end{equation}

\begin{comment}
In this case, optimization~\eqref{opt: crm routing conic} reduces to the following one:
\begin{equation*}
    \begin{array}{ll}
    \underset{x, y, z, w, \lambda, \nu, h}{\mbox{maximize}} & \sum_{k=1}^{n_c} \psi_k(x_k) \\
       \, \mbox{subject to}  & Ey=0_{n_n}, \, x = Hz,\, Fz\leq y,\\
       &Ky\leq (1+h_i)c^i,\, y\leq (1+h_i)d^i,\\
    &y\geq 0_{n_l}, \, z\geq 0_{n_r},\, \,  h\geq 0_N,\,\lambda\geq 0,\\
    & w-\nu1_N\geq h,\, -\nu+2\delta\lambda + p^\top w\leq \epsilon,\\
    &  \norm{w }_\infty\leq \lambda\, i\in[1, N].
    \end{array}
\end{equation*}
\end{comment}